  \documentclass[11pt,letterpaper]{amsart}
\usepackage{amssymb,amsmath,amsthm}
\usepackage{commath}
\usepackage{mathtools}

\usepackage{mathrsfs}
\usepackage{accents}

\usepackage[utf8]{inputenc}

\usepackage{caption}
\usepackage{subcaption}

\usepackage[hyphens]{url}
\usepackage[colorlinks]{hyperref}
\hypersetup{
    allcolors = black
}
\usepackage[hyphenbreaks]{breakurl}

\usepackage{xcolor}

\usepackage{enumitem}
\makeatletter
\def\namedlabel#1#2{\begingroup
    #2
    \def\@currentlabel{#2}
    \phantomsection\label{#1}\endgroup
}

\newtheorem{theorem}{Theorem}[]
\newtheorem{corollary}{Corollary}
\newtheorem{lemma}{Lemma}

\newtheorem{proposition}{Proposition}

\newtheorem{defi}[]{Definition}

\renewcommand{\geq}{\geqslant}
\renewcommand{\leq}{\leqslant}

\newcommand{\Ccal}{\mathcal{C}}

\newcommand{\R}{\mathbb{R}}
\newcommand{\C}{\mathbb{C}}
\newcommand{\Z}{\mathbb{Z}}
\newcommand{\NN}{\mathbb{N}}
\newcommand{\eps}{\varepsilon}

\renewcommand{\O}{\Omega}
\newcommand{\la}{\langle}
\newcommand{\ra}{\rangle}

\newcommand{\ac}{\mathring{a}}
\newcommand{\fc}{\mathring{f}}

\newcommand{\dd}{\, \mathrm{d}}

\newcommand{\bracenom}{\genfrac{\lbrace}{\rbrace}{0pt}{}}

\makeatletter
\@namedef{subjclassname@2020}{
  \textup{2020} Mathematics Subject Classification}
\makeatother

\author[G. Ambrus]{Gergely Ambrus}\thanks{Corresponding author: Gergely Ambrus,
{\em Bolyai Institute, University of Szeged, Hungary, and Alfréd Rényi Institute of Mathematics,  Budapest, Hungary.} e-mail address: \texttt{ambrus@renyi.hu}\\[3 pt]
The article is based on original work, and it have not been published elsewhere in any form or language.}
\author[A. Csiszárik]{Adrián Csiszárik}
\author[M. Matolcsi]{Máté Matolcsi}
\author[D. Varga]{Dániel Varga}
\author[P. Zsámboki]{Pál Zsámboki}

\title{The density of planar sets avoiding unit distances}

\date{\today}



\begin{document}

\begin{abstract}
By improving upon previous estimates on a problem posed by L. Moser, we prove a conjecture of Erdős that the density of any measurable planar set avoiding unit distances is less than $1/4$. Our argument implies the upper bound of $0.2470$.
\end{abstract}

\maketitle

\section{Sets avoiding unit distances}

\label{sect_intro}

How `dense' can a set in $\R^d$ be, if it contains no pairs of points being unit-distance apart? This natural question, raised by Leo Moser in the early 1960s (see \cite{Cr67, Er85}), belongs to a topic known as the {\em combinatorics of unit distances in geometry}. It is a close relative of the famous Hadwiger--Nelson problem posed by Nelson in 1950,
who asked for the minimum number of colors that may be assigned to the points of  $\R^d$ so that points at distance 1 apart are colored differently. This question has enjoyed a surge of interest in the last years, triggered by the breakthrough result of de Grey \cite{deG18} showing that the plane is not 4-colorable, and has been further investigated by the PolyMath16 project \cite{PM22}.  Given a proper coloring, it is readily seen that each color class must avoid unit distances -- hence we arrive at the main object of the present paper.

To make the formulation precise, some definitions are due. A graph with vertices in $\R^d$ is called a {\em unit distance graph} if two vertices are adjacent if and only if  they are at Euclidean distance~1 apart. The {\em unit distance graph of $\R^d$} is obtained by taking all points of $\R^d$ as vertices, and  its chromatic number is denoted by $\chi(\R^d)$. That $\chi(\R^2) \leq 7$ is shown by a periodic coloring of cells of a hexagonal lattice, constructed by  Hadwiger~\cite{Ha45}. For over half a century, the best lower bound had remained $\chi(\R^2) \geq 4$. This is implied by a 7-vertex unit distance graph, the {\em Moser spindle} $M_7$ (Figure~\ref{fig_moser}), constructed by Leo and William Moser \cite{MoM61}, whose chromatic number is 4 (we note that the estimate was already known to Nelson, cf. \cite{So09}). Recently, de Grey~\cite{deG18} proved the breakthrough result $\chi(\R^2) \geq 5$. An alternate proof was given by Exoo and Ismailescu~\cite{ExI20}, while various simplifications and generalizations have been subject to the PolyMath16 project~\cite{PM22}.

\begin{figure}
  \centering
  \includegraphics[height=.25\linewidth]{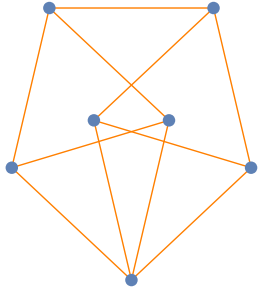}
\caption{The Moser spindle $M_7$ \cite{MoM61}}
\label{fig_moser}
\end{figure}

Imposing the restriction on the color classes to be measurable  leads to the notion of {\em measurable chromatic number of $\R^d$}, denoted by $\chi_m(\R^d)$. Falconer~\cite{Fa81} proved that $\chi_m(\R^d) \geq d+3$ for each $d \geq 2$ -- in particular, $\chi_m(\R^2) \geq 5$. Our result  provides an alternative proof to this statement. Moreover, it also yields that one cannot cover more than $0.988$ fraction of the plane with 4 colors.

A set $A \subset \R^d$ which is independent in the unit distance graph is said to {\em avoid unit distances}, or to be {\em $1$-avoiding} or {\em unit-distance free}.
Assume that $A\subset \R^d$ is measurable. The {\em upper density} of $A$, denoted by $\overline{\delta(A)}$, is defined by
\[
\overline{\delta(A)} = \limsup_{R \rightarrow \infty} \frac{\lambda_d(A \cap B_d(x, R))}{\lambda_d(B_d(x, R))}\,,
\]
where $x \in \R^d$ is fixed, $\lambda_d$ is the $d$-dimensional Lebesgue measure, and $B_d(x, R)$ denotes the $d$-dimensional ball of radius $R$ centered at $x$. The definition is valid since the upper density is known to be independent of the choice of $x\in \R^d$. In case the limit of the above quantity also exists, we call it the {\em density} of $A$, denoted by $\delta(A)$:
\[
\delta(A) = \lim_{R \rightarrow \infty} \frac{\lambda_d(A \cap B_d(x, R))}{\lambda_d(B_d(x, R))}.
\]
This is again independent of $x$.

Let $m_1(\R^d)$ denote the supremum of the upper densities of  unit-distance free, measurable sets in $\R^d$:
\begin{equation}\label{m1def}
m_1(\R^d) = \sup \{ \overline{\delta(A)} \, : \, A\subset \R^d \textrm{ is 1-avoiding and measurable}  \}.
\end{equation}
 The connection to measurable chromatic numbers is provided by the simple fact $\chi_m(\R^d) \geq 1/ m_1(\R^d)$.

 Determining $m_1(\R^2)$ was asked and popularized by Leo Moser \cite[Problem 25]{Mo66}.
 The easiest non-trivial lower bound is obtained by placing open circular discs of radius $1/2$ at the regular hexagonal lattice generated by two vectors of length $2$ enclosing angle $\pi/3$. This is a 1-avoiding set of density $\pi / (8 \sqrt{3}) \approx 0.2267$. A slight improvement was achieved by Croft~\cite{Cr67}. His construction (see Figure~\ref{fig_Croft}) is a lattice arrangement of a {\em tortoise}, that is, the intersection of an open disc of radius $1/2$ with an open regular hexagon of height $x<1$. Placing copies of this tortoise centred at each point of the regular hexagonal lattice with basis vectors of length $1+x$ results in a unit-distance free set. Its density proves to be maximal at $x=0.96553\ldots$ yielding $m_1(\R^2) \geq 0.22936$. 
 To this date, no better construction has been given.

\begin{figure}
\centering
\begin{subfigure}{.49\linewidth}
  \hspace{-30pt}
  \includegraphics[width=1.23\linewidth]{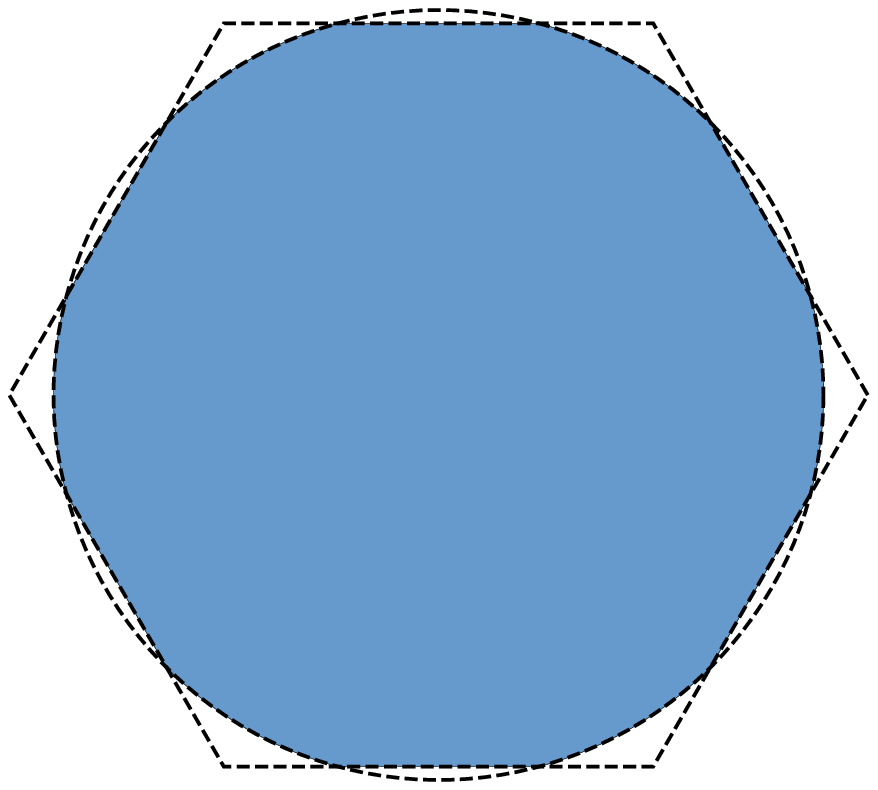}
\end{subfigure}
\begin{subfigure}{.49\linewidth}
  \hspace{-40pt}
  \includegraphics[width=1.23\linewidth]{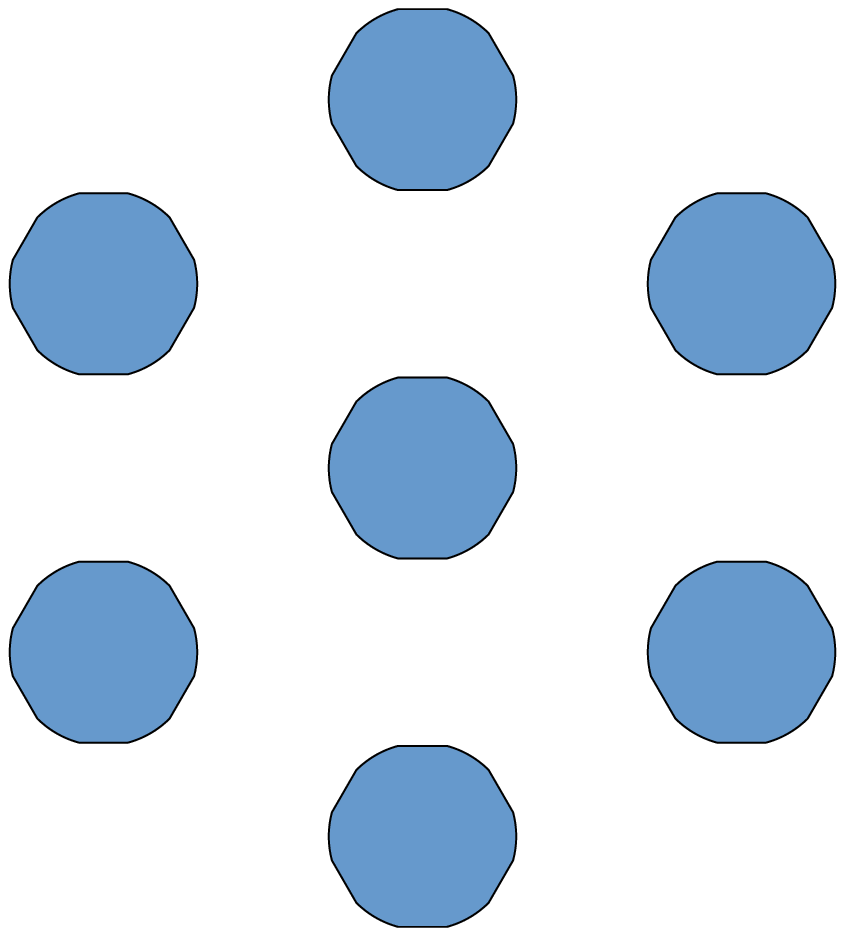}
\end{subfigure}
\caption{Croft's tortoise construction \cite{Cr67}}
\label{fig_Croft}
\end{figure}

 Larman and Rogers \cite{LaR72} formulated a variant of the question restricted to sets contained in the $d$-dimensional unit ball $B_d$: they conjectured that the Lebesgue measure of a 1-avoiding closed set $A \subset B_d$ is less than $1/2^d$ times the Lebesgue measure of the ball. According to their comment added in proof, `In fact this conjecture was made some years ago by Leo Moser.' 
If true, this would have implied that  $m_1(\R^d) \leq 1/2^d$ by averaging. However, the conjecture turned out to be false for all $d\ge 2$, as demonstrated by concrete examples in \cite{deOV19}.

Regarding the maximal density in the whole plane, 
P. Erdős wrote in \cite{Er85}: `It seems very likely that $m_1(\R^2)$ is less than $1/4$.' The main goal of this paper is to prove this conjecture. 

 Let $A$ be a planar, measurable  unit-distance free set with density close to $m_1(\R^2)$ (with a slight abuse of terminology we will assume that the density of $A$ is {\it equal} to $m_1(\R^2)$). The easiest non-trivial upper bound on $m_1(\R^2)$ is $1/2$, which is easily implied by the fact that given an arbitrary unit vector $u$, the sets $A$ and $A+u$ must be disjoint. Taking $u$ and $v$ to be the side vectors of a unit regular triangle emanating from a given vertex, we immediately deduce that $A$, $A + u$ and $A + v$ are also pairwise disjoint, hence $\delta(A) \leq 1/3$. These are the simplest instances of upper bounds based on {\em independence numbers}. Given a unit distance graph $G$, let $\alpha(G)$ denote its independence number. In notation, we sometimes identify $G$ with its vertex set $V(G)$, and simply write $v \in G$ instead of $v \in V(G)$. Consider the set of translates of $A$ by all location vectors of the vertices of $G$. Observe that any point $x \in \R^2$ may be covered by at most $\alpha(G)$ of these translates, since otherwise there would be two connected vertices $u,v \in G$ such that $x \in (A+u) \cap (A+v)$ --- but this latter is equivalent to the fact that both $x - u$ and $x - v$ are points of $A$, which is impossible since $A$ avoids unit distances. Hence, $m_1(\R^2) \leq \frac{\alpha(G)}{|G|}$. Taking $G$ to be the Moser spindle yields the upper bound of $2/7 \approx 0.2857$. 

 There are  two essentially different approaches for giving upper bounds on $m_1(\R^2)$. One direction, which generalizes the idea above, is to search for unit distance graphs with large {\em fractional chromatic number}, which is defined as follows (see e.g. \cite{CrR17}). Given a graph $G$,
we assign to each independent set in $G$ a non-negative weight, such that each
vertex appears in independent sets with weights summing to at least $1$. The
fractional chromatic number $\chi_f(G)$ is the minimum total sum of weights on the independent sets satisfying this condition (for an infinite graph $G$, take the supremum of $\chi_f(H)$, where $H \subset G$ is a finite subgraph). The formal definition will be given in Section~\ref{sect_fracchrom}.  For instance, the fractional chromatic number of the Moser spindle is $7/2$. The fractional chromatic number of the unit distance graph on $\R^2$, $\chi_f(\R^2)$, was first studied by Fisher and Ullmann~\cite{FiU92}, followed by a sequence of steadily increasing lower bounds. Clearly, $\chi_f(G) \leq \chi_f(\R^2)$ holds for any unit distance graph $G$ in the plane.

The crucial link between $m_1(\R^2)$ and $\chi_f(\R^2)$ is provided by the estimate $m_1(\R^2) \leq 1/\chi_f(\R^2)\leq1/\chi_f(G)$  for any unit distance graph $G$ in the plane (see e.g. \cite[Section 3.6]{SchU11} or Corollary~\ref{cor_fcn_bound} below).  In particular, substituting the Moser spindle $M_7$ for $G$ yields the bound $2/7$ we have seen above. In 2018, Bellitto, P\^{e}cher, and S\'{e}dillot \cite{BePS21} utilized this method to find a unit distance graph $G$ on 607 vertices with $\chi_f(G) \geq 3.8991$ which yields $m_1(\R^2) \leq 0.2565$. In 2022, J. Parts~\cite{Pa22} announced the existence of a graph $H$ on 1057 vertices with $\chi_f(H) \geq 3.9898$. If verified (and published), this would have yielded the strongest estimate on $m_1(\R^2)$ prior to the present work.

The other direction stems from a combination of harmonic analysis and linear optimization. The main idea here is to exploit the fact that the auto-correlation function of $A$ defined by $f(x) = \delta(A \cap (A-x))$ is (completely) positive definite. Clearly, $\delta(A) = f(0)$. This value is set as the target function of a linear program whose variables are the Fourier coefficients of $f(x)$, with constraints stemming from elementary combinatorial relations on small unit distance graphs or point sets.  The work of Székely~\cite{Sz84}, which yields an upper bound of $12/43 \approx 0.2791$, acts as a forerunner to this approach: he utilized estimates on $f(x)$ generated by three different point sets. In Section~\ref{sect_linprog} we are going to illustrate  how his estimates fit into our framework. The next step in this direction was taken by de Oliveira Filho and Vallentin~\cite{deOV10}, who --- applying constraints corresponding to three regular triangles --- improved the estimate to $0.2684$. Involving more subtle constraints, Keleti, Matolcsi, de Oliveira Filho and Ruzsa \cite{KeMOR16} showed that $m_1(\R^2) \leq 0.2588$. Finally, by refining the arguments further and including linear constraints generated by the study of triple correlations, Ambrus and Matolcsi~\cite{AmM22} proved that $m_1(\R^2) \leq 0.2544$, which, prior to the present paper,  has been the strongest published estimate. The advantage of this method is that the arising upper bound is guaranteed to converge to $m_1(\R^2)$ as more and more complete positivity constraints on $f$ are specified (see \cite{DeOV22} for details).  

Our method is a common generalization of the two approaches. We prove the following bound, which finally breaks the barrier of $1/4$.

\begin{theorem}\label{thm}
Any Lebesgue measurable, 1-avoiding planar set has upper density at most $0.2470$.
\end{theorem}

We conclude the section by listing a number of related results.
The general characterization of the maximum-density of distance-avoiding sets as the optimal solutions of a convex optimization problem was established by DeCorte, de Oliveira Filho and Vallentin~\cite{DeOV22}. It was proven in \cite{KeMOR16} that the density of a planar unit-distance avoiding set {\em with block structure} is less than  $1/4$. We do not list higher dimensional estimates here (see e.g. \cite{BaPT15, deOV10}).
Further historical details related to unit distance problems may be found in the books of Brass, Moser and Pach~\cite{BrMP05} and Soifer~\cite{So09} as well as the survey article of Székely~\cite{Sz02}.

\section{Inclusion-exclusion constraints}

\label{sect_IE}

Starting with the work of Székely~\cite{Sz84}, a main source of estimates on $m_1(\R^2)$ has been the {\em autocorrelation function} of sets avoiding unit distances. This is defined as follows. Let $A \subset \R^2$ be measurable. The autocorrelation function $f: \R^2 \rightarrow \R$ of $A$ is given (assuming that the densities in question exist) by
\begin{equation}\label{autof}
f(x) = \delta (A \cap (A-x)).
\end{equation}
Clearly, $\delta(A) = f(0)$, and $A$ being unit-distance free implies the condition that
\begin{equation}\label{f_unitdistavoid}
f(x)=0 \textrm{ for all unit vectors }x \in \R^2.
\end{equation}
In order to use Fourier analytic tools we assume that $A$ is periodic with respect to a lattice $L \subset \R^2$, i.e. $A = A+L$. Under this assumption, the density of $A$ exists provided it is measurable. Moreover, $m_1(\R^2)$ may be approximated arbitrarily well by densities of 1-avoiding, measurable, periodic sets  \cite{deOV10, KeMOR16}. Therefore, we may indeed impose the periodicity assumption on $A$ with no effect on the density estimates on $m_1(\R^2)$. 

The estimates of \cite{Sz84, deOV10, KeMOR16, AmM22} essentially relied on elementary density estimates stemming from planar point sets and unit distance graphs of steadily increasing complexity. Roughly speaking, these are derived from a reduced version of the inclusion-exclusion principle (sometimes in disguise). We utilize the {\em complete} inclusion-exclusion formula, which induces no  loss in the density estimates. This idea was first suggested in \cite[Remark 3.3]{KeMOR16}, but has not been explored since. As we will see in Corollary~\ref{cor_fcn_bound}, this leads to a reformulation of the bound coming from the fractional chromatic number.

\medskip
 For a positive integer $n$, write $[n] = \{ 1, \ldots, n\}$ and $\bracenom{n}{2} = \{\{i,j\}: \  i,j \in [n], i \neq j\} $. Let
 $\sigma(n) = \{ \pm 1 \}^n$, the set of $n$-tuples of signs of the form $\eps = ( \eps_1, \ldots, \eps_n )$. For a given $\eps \in \sigma(n)$, introduce
 \[
 I(\eps) = \{ i \in [n]:\  \eps_i = 1 \}.
 \]
  Moreover, for any $I \subset [n]$, let
\[
\sigma(n; I) = \{\eps \in \sigma(n):\ I \subset I(\eps) \} = \{\eps \in \sigma(n):\ \eps|_I = 1 \}.
\]
When the cardinality of $I$ is at most 2, we simply write $\sigma(n; i) := \sigma(n; \{i\})$ and $\sigma(n; i,j) := \sigma(n; \{i,j\})$.

For a set $Y \subset \R^2$ and $\nu \in \{ \pm 1 \}$ introduce the notation
\begin{equation}
Y^\nu=\begin{cases}
Y \textrm{, if } \nu=1 \\
Y^c \textrm{, if } \nu=-1
\end{cases}
\end{equation}
where $Y^c = \R^2 \setminus Y$.

Now, let $X = \{ x_1, \ldots, x_n\} \subset \R^2$. For each $\eps\in \sigma(n)$, set
\[
I_X(\eps) = \{ x_i : \  \eps_i =1 \}
\]
and
\begin{equation}\label{ax_def}
a_X(\eps) = \delta \Big( \bigcap_{i=1}^n (A - x_i)^{\eps_i} \Big).
\end{equation}

The numbers $\{a_X(\eps): \eps \in \sigma(n)\}$ are the densities of `atoms' (i.e. components) of the Venn diagram of the set system $\{ A - x_i : \ i \in [n] \}$; that is, for any non-empty $I \subset [n]$,
\begin{equation}\label{delta_restrict}
\delta\Big( \bigcap_{i \in I} (A - x_i) \Big) = \sum_{\eps \in \sigma(n; I)} a_X(\eps).
\end{equation}

Note that since $A$ avoids unit distances, 
\begin{equation*}
\label{aeps0}
    a_X(\eps) =0 \text{ if } I_X(\eps) \text{ is not an independent set in the unit distance graph. } 
\end{equation*}
These atoms will be referred to as {\em non-independent}, while the others are {\em independent atoms.}

\begin{lemma}[Inclusion-exclusion constraints]\label{lemma1}
Let $f$ be the autocorrelation function of a measurable, periodic, 1-avoiding set $A \subset \R^2$, as defined in \eqref{autof}. Furthermore, let $X = \{ x_1, \ldots, x_n\} \subset \R^2$. Then the set of reals $\{ a_X(\eps) : \eps \in \sigma(n) \}$ defined by \eqref{ax_def} satisfy the following properties:
\medskip
\begin{enumerate}
 \setlength\itemsep{1em}
 \item[\namedlabel{ieP}{\rm{(ieP)}}] $a_X(\eps) \geq 0$ for each $\eps \in \sigma(n) $
 \item[\namedlabel{ieI}{\rm{(ieI)}}] $a_X(\eps)= 0$ for each $\eps \in \sigma(n) $ such that $I_X(\eps)$ contains two points at unit distance
 \item[\namedlabel{ieT}{\rm{(ieT)}}] $ \sum_{\eps \in \sigma(n)} a_X(\eps) =1 $
 \item[\namedlabel{ie1}{\rm{(ie1)}}] $\sum_{\eps \in \sigma(n;i) }  a_X(\eps) = f(0)=\delta(A)$ for every $i \in [n]$
 \item[\namedlabel{ie2}{\rm{(ie2)}}] $\sum_{\eps \in \sigma(n;i,j) }  a_X(\eps) = f(x_i - x_j)$ for every $\{i,j\} \in \bracenom{n}{2}$.
\end{enumerate}
\medskip
\end{lemma}

\begin{proof}
Properties (ieP), (ieI) and (ieT) hold trivially. For (ie1) and (ie2), note that by \eqref{delta_restrict}, the sums on the left-hand side are $\delta(A - x_i)$ and $\delta((A - x_i)\cap (A - x_j))$, respectively; by the invariance of the density of a set by translations we readily obtain that these further equal to $\delta(A)$ and $\delta((A - (x_i - x_j))\cap A)$.
\end{proof}

We can give an upper bound on $\delta(A)$ by maximizing $f(0)$ on the feasible set defined by (ieP), (ieI), (ieT), (ie1). Equivalently, after dividing all variables by $\delta(A)$, that is, defining $\tilde{a}_X(\eps) := a_X(\eps) / \delta(A)$, we can get a lower bound on $1/\delta(A)$ for any unit distance avoiding set $A$ by solving the following linear program in the variables $\tilde{a}_X(\eps)$:

\begin{equation}
\begin{aligned}\label{eq_FCN_LP}
\text{minimize}  \sum_{\eps \in \sigma(n)} &\tilde{a}_X(\eps) \text{ \ subject to:} \\
 &\tilde{a}_X(\eps) \geq 0 \text{ for each } \eps \in \sigma(n) ;\\
 &\tilde{a}_X(\eps) = 0 \text{ for each } \eps \in \sigma(n) \text{ for which } I_X(\eps) \text{ is not independent};\\ 
 &\sum_{\eps \in \sigma(n;i) }  \tilde{a}_X(\eps) = 1 \text{\ for every\ } i \in [n].
\end{aligned}
\end{equation}

As we will see in the next section, this is equivalent to the bound based on fractional chromatic numbers. By also utilizing (ie2), we will go beyond fractional chromatic numbers, and build a connection between the two different approaches used to estimate $m_1(\R^2)$ described in Section~\ref{sect_intro}.

\section{Fractional chromatic numbers}
\label{sect_fracchrom}

Recall the definition of the fractional chromatic number of a finite graph $G$ defined on a vertex set $X$ (for a detailed history see \cite{CrR17}). Denote by $\mathcal {I}(G)$ the set of all independent sets of $G$, and let $\mathcal {I}(G,x)$ be the set of all those independent sets which include the vertex $x \in X$. We will assign non-negative weights $\gamma(S)$ to each $S\in \mathcal{I}(G)$, and for technical reasons we let $\gamma(S)=0$ for any $S\subset X, S\notin \mathcal{I}(G)$.

\begin{defi}\label{def_FCN}
Let $G$ be a finite graph on a vertex set $X$. The {\em fractional chromatic number} of $G$, denoted by $\chi_f(G)$, is defined as 
\begin{equation} \label{agi}
\chi_f(G)=\min \sum _{S\subset X} \gamma(S)
\end{equation}
subject to $\gamma(S) \geq 0$ for every $S\in {\mathcal {I}}(G)$, $\gamma(S) = 0$ for every $S\notin {\mathcal {I}}(G)$ and 
\begin{equation} \label{agix}
\sum_{S\in\mathcal{I}(G,x)} \gamma(S) \geq 1
\end{equation}
for every vertex $x \in X$. 
\end{defi}

Note that restricting the variables $\gamma(S)$ to be integers leads to the notion of the (classical) chromatic number of $G$. Also, for an infinite graph $G'$ it is customary to define $\chi_f(G')=\sup\{\chi_f(G): \ G\subset G', G \text{ \ is finite}$\}. 

\medskip

We next show that equality may be required to hold in \eqref{agix}.

\begin{lemma} \label{lemma_fcn}
For a finite graph $G$, the value of $\chi_f(G)$ is also given as the solution of the modified linear program of Definition~\ref{def_FCN} obtained by  replacing \eqref{agix} with 
\begin{equation} \label{agixeq}
\sum_{S\in\mathcal{I}(G,x)} \gamma(S) = 1
\end{equation}
for each $x \in G$.
\end{lemma}

\begin{proof}
Take a weight system $\{ \gamma(S): \ S \in \mathcal{I(G)}  \}$ which constitutes a solution of \eqref{agi}, moreover, it minimizes 
\begin{equation}\label{totalsumagi}
\sum_{x \in X} \left(-1+\sum_{S\in\mathcal{I}(G,x)} \gamma(S)\right).
\end{equation}
By \eqref{agix}, the above sum is guaranteed to be non-negative; we are going to show that it equals 0, which suffices for the proof of Lemma~\ref{lemma_fcn}. 

Assume on the contrary that for a vertex $x \in G$, $\sum_{S\in\mathcal{I}(G,x)} \gamma(S) = 1+\nu$ holds with $\nu >0$. Note that the minimality condition \eqref{agi} and condition \eqref{agix} ensure that $\gamma(\{x\}) \leq 1$. Hence, there exists an independent set $S_0 \in \mathcal{I}(G)$ with $x \in S_0$, $S_0 \neq \{x\}$ and $\gamma(S_0)>0$. Introduce $\mu = \min\{ \nu, \gamma(S_0) \}$. 
Define a modified weight system $\widetilde{\gamma}$ by setting $\widetilde{\gamma}(S_0)=\gamma(S_0) - \mu$, $\widetilde{\gamma}(S_0 \setminus \{x\})= \gamma(S_0 \setminus \{x\}) + \mu $, and leaving the other weights unchanged. Note that 
\[
\sum_{S\in\mathcal{I}(G)} \widetilde{\gamma}(S) = \sum_{S\in\mathcal{I}(G)} \gamma(S),
\]
and 
\[
\sum_{S\in\mathcal{I}(G,y)} \widetilde{\gamma}(S) = 
\sum_{S\in\mathcal{I}(G,y)} \gamma(S) 
\]
holds for every $y \in G$, $y \neq x$, while
\[
\sum_{S\in\mathcal{I}(G,x)} \widetilde{\gamma}(S) =- \mu+ \sum_{S\in\mathcal{I}(G,x)} \gamma(S)  \geq - \nu+\sum_{S\in\mathcal{I}(G,x)} \gamma(S)  =1 .
\]
Thus, $\sum_{S\in\mathcal{I}(G,x)} \widetilde{\gamma}(S) < \sum_{S\in\mathcal{I}(G,x)} \gamma(S) $  contradicts the minimality of \eqref{totalsumagi}.
\end{proof}

As a direct consequence we obtain the following (cf. \cite[Section 3.6]{SchU11}):

\begin{corollary} \label{cor_fcn_bound}
For any measurable, periodic 1-avoiding set $A\subset \R^2$, and any finite unit-distance graph $G$ in the plane, we have  $1/\delta(A) \geq \chi_f(G)$.  Therefore $m_1(\R^2) \leq 1/\chi_f(G)$  and, taking supremum over $G$,  $m_1(\R^2) \leq 1/\chi_f(\R^2)$. 
\end{corollary}

\begin{proof}
Let $G$ be a finite unit distance graph in the plane with vertex set $X$. When calculating the fractional chromatic number of $G$ , we can assume equality in \eqref{agix}, by Lemma~\ref{lemma_fcn}. This leads to the same  linear program as in \eqref{eq_FCN_LP}, with the correspondence between the variables of \eqref{agix} and \eqref{eq_FCN_LP} being given as $\tilde{a}_X(\eps) = \gamma(I_X(\eps))$.

\medskip

As the linear program \eqref{eq_FCN_LP} gives a lower bound on $1/\delta(A)$, we get  $1/\delta(A) \geq \chi_f(G)$.
\end{proof}

\section{Averaging}
\label{sect_averaging}

We will apply a standard averaging technique, which proves to be a key step. Recall that $A$ is a measurable, periodic, 1-avoiding planar set, that we keep fixed from now on (also meaning that $A$ is not varied by the isometries studied below). As before, let $X = \{ x_1, \ldots, x_n\} \subset \R^2$ be a finite set in the plane. Note that the relations of Lemma~\ref{lemma1} remain valid for every image of the subset $X$ under an isometry $\varphi$ of $\R^2$ with the values of  $a_{\varphi(X)}(\eps)$ being provided by \eqref{ax_def}. In particular, we may take the average of these relations over the orthogonal group $O(2)$ of the plane. Let $\mu$ denote the Haar probability measure on $O(2)$, and introduce the notations
\begin{equation}\label{a_av}
\ac_X(\eps) = \int_{O(2)}  \delta \Big( \bigcap_{i=1}^n (A - \varphi(x_i))^{\eps_i} \Big) \dd\mu(\varphi)
\end{equation}
for $\eps \in \sigma(|X|)$ and
\begin{equation*}\label{f_av}
\fc(x) = \int_{O(2)}  \delta \big( A \cap  (A - \varphi(x))\big) \dd\mu(\varphi)
\end{equation*}
for $x \in \R^2$. Note that
\begin{equation}\label{f_av2}
  \mathring{f}(x) = \frac 1 {2 \pi} \int_{S^1} f(\xi |x|) \dd \omega(\xi),
\end{equation}
where $\omega$ is the perimeter measure on the unit circle $S^1$. However, this does not generally hold for $\ac_X(\eps)$ since reflections also have to be taken into account.

This averaging process has two-fold benefit.  First, $\mathring{f}$ is radial (i.e. it depends only on $|x|$). Second, the averaging of relation ~\eqref{delta_restrict} leads to an additional set of natural constraints on $\ac_X(\eps)$. Indeed, let us introduce the notation $Y \cong Y'$, whenever $Y, Y'\subset X$ are congruent subsets of $X$. For an index set $I\subset [n]$ let  $X|_I = \{ x_i: \ i \in I \}$, and introduce
\[
\Ccal(X) = \big\{ \{ I, J\}: I, J \subset [n], \ I \neq J, \ X|_I \cong X|_J\big\}.
\]
Whenever $X|_I \cong X|_J$, there exist some orthogonal transformation $\varphi$ and a translation $\tau$ such that $X|_J=\varphi(\tau(X|_I))$. By equation \eqref{ax_def} we have   $a_X(\varepsilon)=a_{\tau(X)}(\varepsilon)$, and the averaging of equation  ~\eqref{delta_restrict}  over $\varphi$ leads to the constraint
\medskip
\begin{equation}\label{ieC}
\tag{ieC}
\sum_{\eps \in \sigma(n; I) }  \ac_X(\eps) =  \sum_{\eps \in \sigma(n; J) }  \ac_X(\eps) \textrm{ for every } \{I, J\} \in  \Ccal(X).
\end{equation}

\medskip

{\em Remark.} This averaging process, somewhat counter-intuitively,  does not result in any loss in the estimates. We do not give a rigorous proof of this fact (as we do not need it in the sequel), but only a heuristic argument, as follows. Partition the plane into large squares of size $Q\times Q$, $S(i,j)=[Qi, Q(i+1)]\times [Qj, Q(j+1)]$, where $Q$ is a huge number.  For each $(i,j)$ select an orthogonal transformation $\varphi$ at random (uniformly with respect to the Haar measure), and copy $\varphi(A)\cap S(0,0)$ into the square $S(i,j)$. Finally, delete a strip of sidelength 1 at the boundaries of the squares to make sure that no unit distances are created at the boundaries. The resulting set will have density close to that of $A$ (because the deletions are negligible), and the density correlations defined in \eqref{delta_restrict} will be effectively invariant with respect to orthogonal transformations, up to a small error, as long as the size of the set $I$ is small compared to $Q$.

It is now natural to extend the set of constraints used to define the fractional chromatic number with a new congruence constraint based on  \eqref{ieC}, which leads to a new notion.

\begin{defi}\label{def_GFCN}
Let $G$ be a finite graph with vertex set $X$ in the plane. The {\em geometric fractional chromatic number} of $G$, denoted by $\chi_{gf}(G)$, is defined as 
\begin{equation*} 
\chi_{gf}(G)=\min \sum _{S\subset X} \gamma(S)
\end{equation*}
subject to $\gamma(S) \geq 0$ for every $S\in {\mathcal {I}}(G)$, $\gamma(S) = 0$ for every $S\notin {\mathcal {I}}(G)$,   
\begin{equation*}
\sum_{S\in\mathcal{I}(G,x)} \gamma(S) = 1
\end{equation*}
for every vertex $x \in G$, and  
\begin{equation*} 
\sum_{S\subset T}\gamma(T) = \sum_{S'\subset T'}\gamma(T') 
\end{equation*}
whenever $\{ S, S' \} \in \Ccal(X)$, i.e.  $S, S'$ are congruent subsets of $G$. 
\end{defi}
Clearly, $\chi_{gf}(G) \geq \chi_f(G)$ holds. 
We also note that the definition may naturally be extended to graphs realized in $\R^d$ with $d \geq 2$, and that for infinite graphs $G'$ we can define $\chi_{gf}(G')=\sup\{ \chi_{gf}(G): \ G\subset G', \ G \ \text{is fnite}\}$.

\medskip

The analogue of Corollary~\ref{cor_fcn_bound} reads as follows. 

\begin{corollary}\label{cor_gfcn}
For any measurable, periodic 1-avoiding set $A\subset \R^2$, and any finite unit-distance graph $G$ in the plane we have  $1/\delta(A)\ge \chi_{gf}(G)$. Therefore $m_1(\R^2) \leq 1/\chi_{gf}(G)$,  and taking supremum over $G$ we get  $m_1(\R^2) \leq 1/\chi_{gf}(\R^2)$. 
\end{corollary}

\begin{proof}
By the relations (ieP), (ieI), (ie1) and (ieC), the quantities $\ac_X(\eps)/\delta(A)=\gamma(I(\eps))$, as defined by \eqref{a_av}, constitute a solution to the linear program in Definition \ref{def_GFCN}. Furthermore,  $\sum_{S\subset X} \gamma(S) = 1/\delta(A)$ by (ieT). Therefore, $1/\delta(A)\ge \chi_{gf}(G)$.
\end{proof}

\medskip

\noindent
{\em{Remark.}} These notions lead to an interesting connection between measurable and non-measurable quantities of the unit distance graph of $\R^2$. The details of the discussion below are not essential for the purposes of this paper, and will therefore be elaborated only in a follow-up publication.

Let $\alpha_1(\R^2)=\inf\{\alpha(G)/|G|: \mathrm{G \ is \ a \ finite \ UDG}\}$. One can think of $m_1(\R^2)$ as the measurable independence ratio, while $\alpha_1(\R^2)$ as the non-measurable independence ratio of~$\R^2$. Clearly,  $m_1(\R^2)\le \alpha_1(\R^2)$. 
Also, it is possible to show that $\alpha_1(\R^2)=1/\chi_f(\R^2)$. 

Furthermore, for any finite unit distance graph $G$ we trivially have  
 $\chi_f(G)\le \chi_{gf}(G)$, and hence   $\chi_{f}(\R^2)\le \chi_{gf}(\R^2)$. It turns out that by increasing the size of the graphs and applying a careful averaging argument one can prove the interesting conclusion 
$\chi_{f}(\R^2)= \chi_{gf}(\R^2) $. 

Combining these two observations we have $m_1(\R^2)\le \alpha_1(\R^2)=1/\chi_f(\R^2)=1/\chi_{gf}(\R^2)$. 

\medskip

As of now, our numerical search did not result in any graph $G \subset \R^2$ with $\chi_{gf}(G)>4$, with the record holder being $\chi_{gf}(G)= 3.9954$. Based on this numerical evidence, we venture to conjecture that $\chi_f(\R^2)= \chi_{gf}(\R^2) = 4$. This would imply that $\alpha_1(\R^2)=1/4$ and $m_1(\R^2)\le 0.247$ hold simultaneously, providing an interesting example of a natural infinite graph for which the measurable and the non-measurable variants of a parameter have different values.

\section{Fourier analytical tools}
\label{sect_fourier}
We are going to specify the autocorrelation function $f$ via its Fourier coefficients. 
A fairly detailed discussion of the procedure is contained in \cite{deOV10} and \cite{KeMOR16}, and we only include a summary here to make the present paper self-contained. For further details about these standard tools in harmonic analysis, we direct the reader to the book of Katznelson~\cite{Ka68}).

We remind the reader that the 1-avoiding set $A$ is assumed to be periodic with a period lattice $L$. This enables us to perform a Fourier expansion of $f(x)=\delta(A\cap (A-x))$ in the Hilbert space $L^2(\R^2/L)$.

The Fourier coefficients $\hat{f}(u)$ are defined, as usual, by
\[
\hat{f}(u) = \la f, \chi_u \ra
\]
where $\chi_u(x) = e^{i u \cdot x}$, $u \in 2 \pi L^*$, where $L^\ast$ is the dual lattice  of $L$:  $L^\ast =\{u\in \R^2 : u \cdot v \in \Z \ \textrm{for all} \ v\in L\}$, and for any
measurable functions $h, g: \R^2 \to \C$, 
\[
\la h,g \ra = \lim_{T \to \infty} \frac 1 {(2T)^2} \int_{[-T, T]^2} h(x) \overline{g(x)} \dd x\ , 
\]
if the limit exists.  Note that $f$ may be written as 
\[
f(x) = \la {\bf 1}_A, {\bf1}_{A - x} \ra,
\]
and hence, by $\widehat{{\bf1}}_{A - x}(u)=\widehat{{\bf1}}_{A}(u) e^{i u \cdot x}$, Parseval's identity yields 
\[
f(x) = \sum_{u \in 2 \pi L^*} |\widehat{{\bf 1}}_A(u)|^2 e^{i u \cdot x}\,.
\]
Thus, the Fourier coefficient $\widehat{f}(u)$ corresponding to the term $e^{i u \cdot x}$ is given by
\begin{equation}\label{f_pos}
  \widehat{f}(u)=|\widehat{{\bf 1}}_A(u)|^2 \geq 0
\end{equation}
for every $u \in 2 \pi L^*$. After averaging as in \eqref{f_av2}, we may express $\fc(x)$ as
\begin{equation}\label{fkor}
\mathring{f}(x) =\sum_{u \in 2 \pi L^*}\widehat{f}(u) \Omega_2( |u| |x|),
\end{equation}
where $\O_2(|x|)$ is the Bessel function of the first kind with parameter 0,  defined as
\[
\O_2(|x|) = \frac 1 {2 \pi} \int_{S^1} e^{ i x \cdot \xi} d \omega(\xi),
\]
for any $x\in \R^2$. Note that $\mathring{f}(x)$ is radial. It is not periodic (unlike $f$), and it is not the autocorrelation function of any nonempty set.

Introducing the notation
\[
\kappa(t) = \sum_{u \in 2 \pi L^* , |u| = t} \widehat{f}(u)
\]
for $t \geq 0$, \eqref{fkor} simplifies to
\begin{equation}\label{fkappa}
\mathring{f}(x) = \sum_{t \geq 0} \kappa(t) \Omega_2( t |x|),
\end{equation}
where the summation is taken for those (countably many) values of $t$ which come up as a length of a vector in $2\pi L^*$. In particular, for $x = 0$ we have that 
\begin{equation}\label{fkappa0}
\delta(A) = f(0)=\mathring{f}(0) = \sum_{t \geq 0} \kappa(t) .
\end{equation}
With this, we have completed the technical preparations required to transform the problem to the linear optimization setting.

\section{Linear programming format}
\label{sect_linprog}

We are going to estimate $\delta(A) = f(0) = \fc(0)$ by virtue of reformulating the problem as a linear optimization program. The constraints of the program correspond to the natural conditions on the Fourier coefficients of $f$ as well as  the inclusion-exclusion constraints generated by a finite set of points $X \subset \R^2$ with cardinality $n$.  Note that by \eqref{fkappa0}, the objective function can be expressed as $\delta(A)= \sum_{t \geq 0} \kappa(t)$.

The (infinitely many) variables of the linear program are $\{ \kappa(t): \ t \geq 0 \}$ and $\{\ac_{X}(\eps): \eps \in \sigma(n) \}$. Note that these latter are auxiliary variables which are not present in the objective function.

The linear program, which we are going to refer to as (LP) from now on, reads as follows -- the constraints are implied by \eqref{f_pos}, \ref{ieP}, \eqref{f_unitdistavoid}, \ref{ieT}, \ref{ie1}, \ref{ie2}, and \eqref{ieC}  and must hold for any set of coefficients derived from a measurable, periodic, planar 1-avoiding set.
\medskip

\renewcommand{\arraystretch}{1.7}
\begin{tabular}{lll}
&\textrm{Maximize} $\sum_{t \geq 0} \kappa(t)$ &\\
&\hspace{1 em} \textrm{subject to} &\\
(CP) & $\kappa(t) \geq 0$  \textrm{for every} $t \geq 0$\\
(IEP) & $\ac_X(\eps) \geq 0$ for each $\eps \in \sigma(n) $\\
(C0) & $ \sum_{t \geq 0} \kappa(t) \Omega_2( t ) =0 $\\
(IET) & $ \sum_{\eps \in \sigma(n)} \ac_X(\eps) =1 $ \\
(IE1) & $\sum_{t \geq 0} \kappa(t) - \sum_{\eps \in \sigma(n;i) }  \ac_X(\eps) = 0$ for every $i \in [n]$ \\
(IE2) & $\sum_{t \geq 0} \kappa(t) \Omega_2( t |x_i - x_j|) - \sum_{\eps \in \sigma(n;i,j) }  \ac_X(\eps) = 0$ for every $\{i,j\} \in \bracenom{n}{2}$\\
(IEC) &
$\sum_{\eps \in \sigma(n; I) }  \ac_X(\eps) -  \sum_{\eps \in \sigma(n; J) }  \ac_X(\eps)=0 \textrm{ for every } \{I, J \} \in \Ccal(X)$.
\end{tabular}
\medskip

\noindent 
By \eqref{fkappa0}, the solution of this linear program gives an upper bound on $\delta(A)$.

\medskip

\noindent
{\em{Remark.}} Note that by means of listing congruence constraints in (LP) for pairs of 1-element and 2-element vertex sets too, it suffices to include (IE1) only for one vertex (we choose  $i=1$), and (IE2) for one pair of vertices for each distance among points of $X$. Also note that (IEP), (C0) and (IE2) imply that $\ac_X(\eps) = 0$ is forced to hold for each non-independent atom, in accordance with (ieI) of Lemma~\ref{lemma1}. Thus, we only have to include independent atoms in the list of variables of (LP).

\medskip

In principle, it would be possible to apply Lemma~\ref{lemma1} to several point sets $X_1, \ldots, X_m$ in order to gain a wider set of linear constraints. Notice, however, that all these constraints are implied by the ones obtained by applying Lemma~\ref{lemma1} and \eqref{ieC} to $\bigcup_{i=1}^m X_i$. Since this yields a stronger estimate,  we  only apply one set $X$.

\medskip 

We also remark that the constraints of (LP) induce the ones in \cite[Theorem 2.6]{Sz84}, \cite[Lemma 3.2]{KeMOR16} and \cite[Lemma 2.1, Lemma 2.2, Lemma 3.1, Lemma 3.2]{AmM22} provided that $X$ contains a congruent copy of all the point sets applied within these --- therefore, our new estimate improves on these earlier ones. We illustrate this with the bound of Székely~\cite{Sz84}. His estimate of  $12/43 \approx 0.2791$ stems from considering the three graphs on the left side of Figure~\ref{fig_szekely} (and, importantly, using the observation that $\fc(t) \leq \fc(0)/2$ for each $t \geq 1/2$ which can be obtained by averaging on circles of radius $t$). Note that the 7-point set $X$ on the right side contains a congruent copy of all three of these. Solving the linear program (LP) associated to this point set $X$  yields the upper bound of $0.2771 < 12/43$ on $\delta(A)$, which improves on the original bound.

\medskip

More generally, in \cite{DeOV22} it is shown how inclusion-exclusion inequalities and subgraph constraints
come from facets of the Boolean quadratic polytope. It is reasonable to expect that some of the constraints in (LP) above also correspond to 
facets of this polytope. However, to extract such information one would need to eliminate the `atomic' variables $\ac_X(\eps)$, and deduce constraints involving the values of the autocorrelation function $\fc(t)$ only. We do not see a direct way of doing this.

\begin{figure}
\centering
\includegraphics[width=\linewidth]{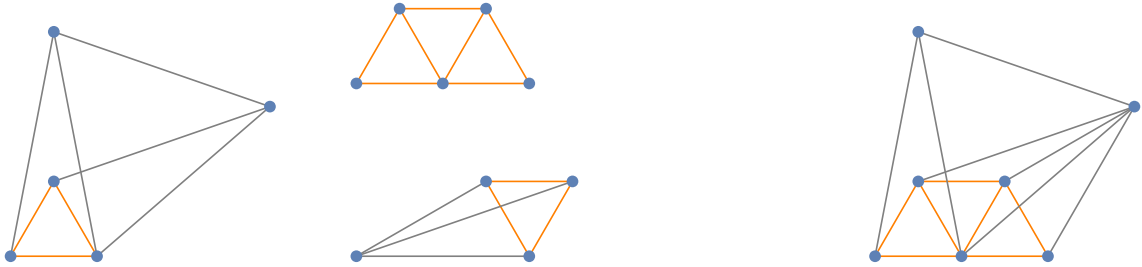}
\caption{{\em Left:} the three graphs used by Székely~\cite{Sz84}.  {\em Right:} a 7-point set $X$ containing all these.}
\label{fig_szekely}
\end{figure}

\medskip

By standard linear programming duality, the solution of the dual program provides an upper bound on the solution of (LP). Note that the dual of a linear program of the form
\[
\textrm{maximize } \mathbf{c}^\top \mathbf{x} \ \textrm{ subject to } A \, \mathbf{x} = \mathbf{b}, \ \mathbf{x}\geq \mathbf{0}
\]
is given by
\begin{equation}\label{dualLP}
\textrm{minimize } \mathbf{b}^\top \mathbf{y} \ \textrm{ subject to } A^\top \, \mathbf{y} \geq  \mathbf{c}.
\end{equation}
Applying this in the present context yields the next statement.

\begin{proposition}\label{prop1}
  Let $X$ be a set of $n$ points in the plane. Suppose that the real numbers $w_0$, $w_T$, $\{ w_1(i): i \in [n] \}$,  $\{ w_2(i,j): \{i,j\} \in \bracenom{n}{2} \}$, $\{ w_c(I,J): \{I,J\} \in \Ccal(X) \}$ are such that  the function defined by
  \begin{equation}\label{W_def}
  W(t) =   w_0 \, \Omega_2(t) + \sum_{i \in [n]} w_1(i) + \sum_{\{i,j\} \in \bracenom{n}{2}} w_2(i,j)   \Omega_2(t |x_i - x_j|)
  \end{equation}
  satisfies $W(t) \geq 1$ for every $t \geq 0$, and the quantity defined by
  \begin{equation}\label{V_def}
  \begin{split}
  V(\eps) = w_T  - \sum_{i \in [n]:\ \eps_i=1} &w_1(i)  -  \sum_{\{i,j\} \in \bracenom{n}{2}:\ \eps_i=\eps_j=1} w_2(i,j) \\ &+ \sum_{\{I,J\} \in \Ccal(X) : \eps \in \sigma(n; I) } w_c(I,J) - \sum_{\{I,J\} \in \Ccal(X) : \eps \in \sigma(n; J) } w_c(I,J)
  \end{split}
  \end{equation}
  satisfies $V(\eps) \geq 0$ for every $\eps \in \sigma(n)$. Then
  \begin{equation}\label{mwT}
  m_1(\R^2) \leq w_T.
  \end{equation}
\end{proposition}

\begin{proof}
This is simply the dual linear program \eqref{dualLP}. Without referring to the general theory of linear programming, one can derive the estimate easily as follows. Assume that the coefficients specified above satisfy the criteria. Let $A$ be a measurable, periodic, planar 1-avoiding set with autocorrelation function $f$. Then, by the properties listed in (LP),
\begin{align*}
    \delta(A) &= f(0) = \fc(0) = \sum_{t \geq 0} \kappa(t) \leq \sum_{t \geq 0} \kappa(t) W(t) \\
    &= w_0 \sum_{t \geq 0} \kappa(t) \Omega_2(t) + \sum_{i \in [n]} w_1(i) \sum_{t \geq 0} \kappa(t)   + \sum_{\{i,j\} \in \bracenom{n}{2}} w_2(i,j)   \sum_{t \geq 0} \kappa(t) \Omega_2(t |x_i - x_j|)\\
    &= 0 + \sum_{i \in [n]} w_1(i) \sum_{\eps \in \sigma(n;i) }  \ac_X(\eps) + \sum_{\{i,j\} \in \bracenom{n}{2}} w_2(i,j)  \sum_{\eps \in \sigma(n;i,j) }  \ac_X(\eps) \\
    &= \sum_{\eps \in \sigma(n)}  \ac_X(\eps) \Big( \sum_{i \in [n]:\ \eps_i=1} w_1(i)  +  \sum_{\{i,j\} \in \bracenom{n}{2}:\ \eps_i=\eps_j=1} w_2(i,j)  \Big)\\
    &\leq \sum_{\eps \in \sigma(n)}  \ac_X(\eps) \Big( w_T + \sum_{\{I,J\} \in \Ccal(X) : \eps \in \sigma(n; I) } w_c(I,J) - \sum_{\{I,J\} \in \Ccal(X) : \eps \in \sigma(n; J) } w_c(I,J)  \Big)\\
    &=w_T +  \sum_{\{I,J\} \in \Ccal(X)} \Big(\sum_{\eps \in \sigma(n; I) }  \ac_X(\eps) -  \sum_{\eps \in \sigma(n; J) }  \ac_X(\eps) \Big)\\
    &=w_T\,. \qedhere
\end{align*}
\end{proof}

\noindent
Note that by means of the Remark following the definition of (LP), it suffices to check that $V(\eps) \geq 0$ holds for each $\eps \in \sigma(n)$ which represents an independent atom. This accelerates the numerical computations greatly. 

\medskip

\section{Breaking the 1/4 barrier}\label{breaking}

Our goal is to search for a point set  $X$ such that the solution of (LP) is as small as possible. However, this linear program has infinitely many variables of the form $\kappa(t): t \geq 0$, which necessitates to use a discrete approximation. Therefore, we only search for the coefficients $\kappa(t_i)$, where $t_i = i \eps_0$ with $i \in \NN$ such that  $t_i \in [0, t_{max}]$, with some $\eps_0>0$ and $t_{max} >0$. For all other values of $t \geq 0$, we set $\kappa(t)=0$. Based on the calculations in Section~\ref{sect_verification}, we set $\eps_0 =0.05 $ and $t_{max}=600$.

Similarly to the preceding works, we utilize a heuristic computer search in order to find a suitable point set.
Initially, we combined the (IE) constraints with the ones used in \cite{AmM22} in order to improve the estimates. Yet, the incrementally stronger set of (IE) constraints quickly outpowered the older ones. Note that even though this would also follow {\em formally} from our previous remarks, the cardinality bound on $X$ stemming from computational complexity makes this observation non-trivial. Nevertheless, based on empirical evidence, we apply a tabula rasa approach and search for constraints arising from a single point set~$X$.

We limit our search to point sets of cardinality $\leq 24$, as this proved to be sufficient to reach our goal of proving Erdős's conjecture.  Note that when defining (LP), not every $\ac_X(\eps)$ needs to be added as a new variable. Indeed, if for $\eps \in \sigma(n)$, $\eps_i = \eps_j = 1$ holds with some $i \neq j$ and $|x_i - x_j| =1$, then $\ac_X(\eps)$ is forced to equal to 0 by (ie2). Therefore, these variables may be omitted from (LP). Hence, sets with many unit distances lead to (LP)s with fewer variables, which helps in practical implementation.

To exploit congruence constraints of the form (IEC), it is desirable that $X$ contains many congruent pairs of subsets. Combining this with the preference for unit distances, we adopt the following strategy. Starting from an initial point set $X_0$ (to be specified later as the Moser spindle), we add points to it one-by-one so that each new point is at distance 1 from at least two points which have already been constructed. Thus, we are searching for a node with a low score in the following tree:
\begin{enumerate}
    \item Each node of the tree corresponds to a finite planar set $X$. The score of each node is the solution of (LP) resulting from the corresponding point set. 
    \item The root is the initial set $X_0$.
    \item Given a node $X_i$, its children are the nodes obtainable from $X_i$ by adding one point: $X_{i+1}=X_i\cup\{x_{i+1}\}$ such that:
    \begin{enumerate}
        \item To have only finitely many choices at each step, we only consider those points $x_{i+1}$ that  have distance 1 to at least two points of $X_i$.
        \item  In light of the error estimates in Section~\ref{sect_verification}, we further restrict the set of accepted points $x_{i+1}$ to those having distance at least 0.1 to all points of $X_i$.
    \end{enumerate}
\end{enumerate}

In order to exploit non-trivial parallel relations between pairs of vertices, we choose the Moser spindle $M_7$ to serve as $X_0$.

For the search algorithm we use beam search, which may be formalized as follows.  
\begin{enumerate}
    \item At each step, we have in hand a collection $\mathscr X_i$ of nodes $X_i$ at a given level.
    \item We initialize with $\mathscr X_0=\{X_0\}$.
    \item We proceed as follows:
    \begin{enumerate}
        \item We let $\widetilde{\mathscr X}_{i+1}$ denote the collection of all children of all nodes in the present collection $\mathscr X_i$.
        \item We calculate the scores for all sets in $\widetilde{\mathscr X}_{i+1}$.
        \item We sort $\widetilde{\mathscr X}_{i+1}$ based on the scores.
        \item We let $\mathscr X_{i+1}$ be the collection of the best (i.e. lowest scoring) 100 children in $\widetilde{\mathscr X}_{i+1}$. (The parameter 100 is referred to as the beam width.)
    \end{enumerate}
\end{enumerate}

The output of the above beam search algorithm is a point cloud specified by algebraic expressions. As explained in  Section~\ref{sect_impl}, we speed up running time by deferring the symbolic computation until we settle on a given point set.

We present a point set $X_{23}$ with 23 vertices that was provided by the computer search described above. Its exact, symbolic coordinates are listed in Table~\ref{table:point set} of the Appendix. We denote the corresponding unit distance graph by $G_{23}$ (see Figure~\ref{fig_x23}). We note that the majority of the vertices have degree at least 3 in $G_{23}$, and all vertices have degree at least 2. Pairs of points of $X_{23}$ determine only 27 distinct distances (this partly accounts for the large number of congruent pairs of vertex sets). As observed earlier, it suffices to include among the variables of (LP) those atoms which correspond to independent sets in the unit distance graph, as the others must have density 0. With this simplification, the linear program defined by $X_{23}$ has 13552 atom variables and 12000 Fourier variables. It consists of 23 (IE1) constraints, 206 (IE2) constraints connecting the Fourier variables to the atom variables, and 5868 (IEC) congruence constraints.

Figure~\ref{fig_fs} depicts the radialized autocorrelation function $\fc$ yielded by $X_{23}$, together with the previous best upper bound given by \cite{AmM22}, and the radialized autocorrelation function of the Croft construction. Note that our upper bound yields an $\fc$ significantly closer to the Croft construction.

The numerical solution of (LP) obtained from $X_{23}$ is $0.24697$. In order to prove an upper bound on $m_1(\R^2)$ which nearly equals this value, we need to apply Proposition~\ref{prop1} and estimate the errors stemming from numerical computations and discrete approximation. This is the goal of the next section.

\begin{figure}
  \centering
  \includegraphics[width=0.42\linewidth]{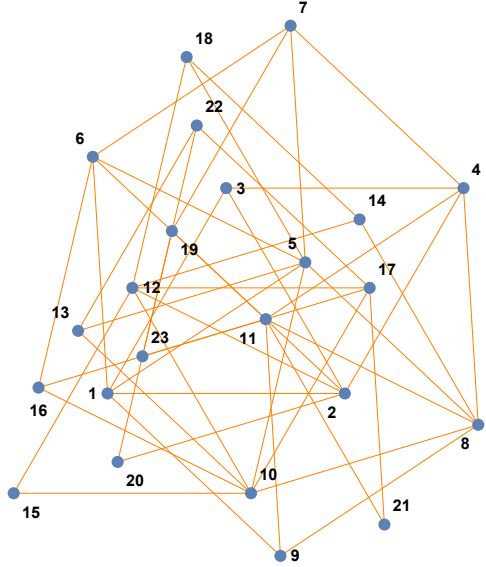}
\caption{The unit distance graph $G_{23}$ defined on the vertex set $X_{23}$.}
\label{fig_x23}
\end{figure}

\begin{figure}
\centering
\includegraphics[width=0.75 \linewidth]{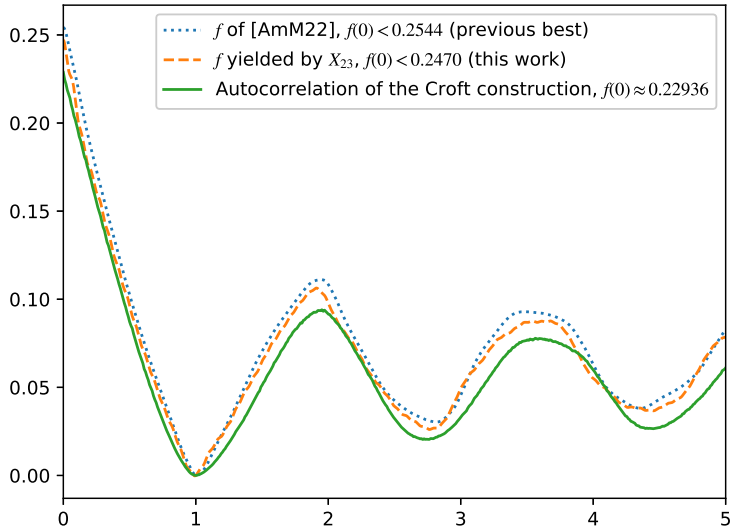}
\caption{Plotting autocorrelation functions on the interval $[0,5]$.}
\label{fig_fs}
\end{figure}

\section{Verification}
\label{sect_verification}

In this section, we address error estimates resulting from the discrete approximation of the linear program (LP) we use. Fix (LP) to be the linear program corresponding to $X_{23}$.
To start with, we set $w_0$, $w_T$, $\{ w_1(i): i \in [n] \}$,  $\{ w_2(i,j): \{i,j\} \in \bracenom{n}{2} \}$, $\{ w_c(I,J): \{I,J\} \in \Ccal(X) \}$ to be the coefficients of the solution vector of the dual of (LP). There are 2350 non-zero coefficients, of which 2321 are of type $ w_c(I,J)$. Due to space limits, we do not list these in the article; however, they are made available at \cite{Web}. The remaining coefficients are listed in Table \ref{table:dual coefficients}.

We are going to modify these slightly so that the conditions on $W(t)$ and $V(\eps)$ of Proposition~\ref{prop1} are satisfied for them, and apply estimate \eqref{mwT}  for this modified set of coefficients.

First, we address $V(\eps) \geq 0$, where $\eps \in \sigma(n)$ and $V(\eps)$ is given by \eqref{V_def}. This is a finite set of inequalities, and high precision calculation ensures that 
\[
V(\eps) \geq -\nu
\]
holds true for each $\eps \in \sigma(n)$ with $\nu = 10^{-5}$.  Thus, by changing $w_T$ to $\widetilde{w}_T = w_T + \nu$ we guarantee that 
$V(\eps) \geq 0$ holds for all instances. 

Second, we estimate $W(t)$ defined by \eqref{W_def} by an argument which is analogous to \cite[Section 3.2]{KeMOR16}. As first observed by Schoenberg~\cite{Sch38},  $\Omega_2(t) = J_0(t)$, the Bessel function of the first kind with parameter 0. Accordingly, let 
\begin{equation}\label{eq_phi}
    \varphi(t) =  w_0 \, J_0(t) + \sum_{i \in [n]} w_1(i) + \sum_{\{i,j\} \in \bracenom{n}{2}} w_2(i,j)    J_0(t |x_i - x_j|)
\end{equation}
with $x_i$ being defined by Table~\ref{table:point set} and the non-zero coefficients $w$ specified by Table \ref{table:dual coefficients}. For small (and not so small) values of $t$, the function $\varphi(t)$ is plotted on Figure~\ref{fig_fi}. Note that, in particular, $w_1(i) = 0$ except for $i = 1$, for which we have $w_1(1) = 1.059383649998022$.

\begin{figure}
\centering
\begin{subfigure}{.49\textwidth}
  \centering
  \includegraphics[width=0.9\linewidth]{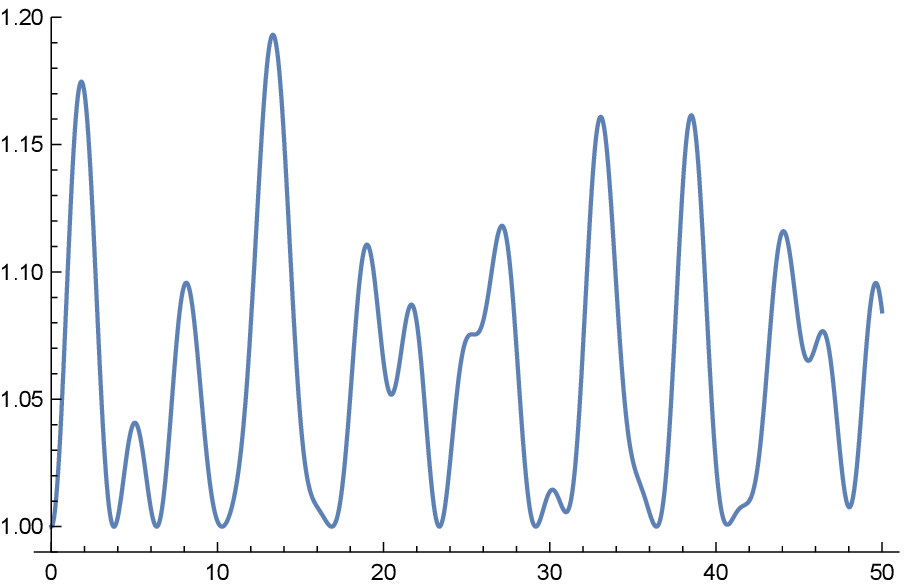}
\end{subfigure}
\begin{subfigure}{.49\textwidth}
  \centering
  \includegraphics[width=0.9\linewidth]{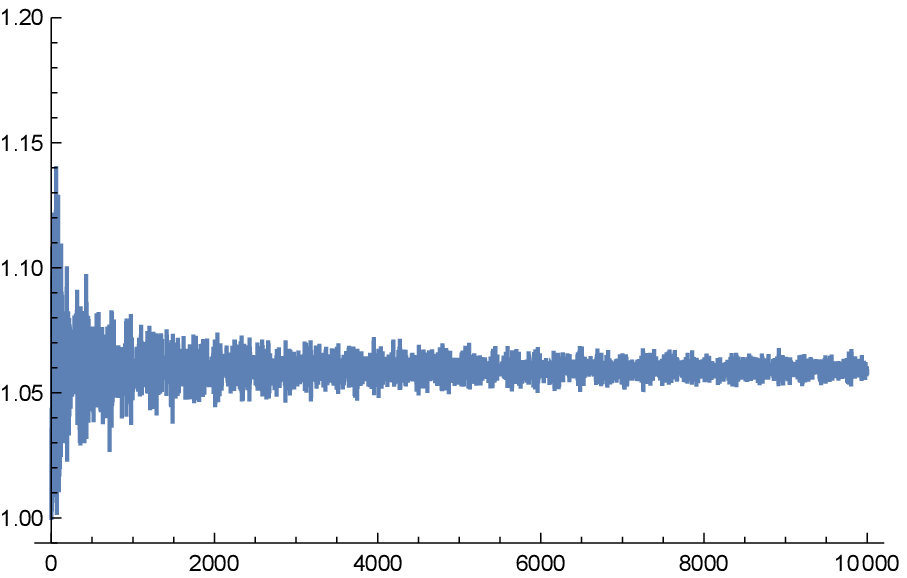}
\end{subfigure}
\caption{The function $\varphi(t)$ plotted on $[0,50]$ ({\em left}) and $[0,10000]$ ({\em right}).}
\label{fig_fi}
\end{figure}

In order to show that $\varphi(t)\geq 1$ for all $t \geq 0$, we first note that 
\[
\lim_{t \to \infty} J_0(t) = 0.
\]
This follows from the asymptotic formula for $J_\alpha$ for $\alpha \geq 0$ (cf. Watson \cite[§7.21, eq. (1)]{Wa22}). Accordingly, 
\[
\lim_{t \to \infty} \varphi(t) = w_1(1)= 1.059383649998022,
\]
hence $\varphi(t) \geq 1$ if $t$ is sufficiently large. We have to provide an explicit threshold here. 

Note that $J_0'(t) = J_1(t)$. Thus, the locations of local extrema of $J_0(t)$ agree with the positive zeroes $j_1 < j_2 < j_3 <\ldots$ of $J_1(t)$. The local extrema of $J_0$ decrease in absolute value (cf. \cite[§15.31]{Wa22}), that is
\[
|J_0( j_1)| > |J_0( j_2)| > |J_0( j_3)| > \ldots
\]
Thus, by \eqref{eq_phi}, on the domain $ t \in [ T , \infty)$ the following estimate holds:
\begin{equation} \label{phi_est}
\varphi(t) \geq w_1(4) - \Big(|w_0| +  \sum_{\{i,j\} \in \bracenom{n}{2}} |w_2(i,j)| \Big) |J_0 (s)|,
\end{equation}
where $s$ is any zero of $J_1$ not exceeding $T \cdot d_{min} $ where
\[
d_{min} = \min_{\{i,j\} \in \bracenom{n}{2}, \ w_2(i,j) \neq 0\ }   |x_i - x_j|.
\]
We restricted our search to point sets with $d_{min} \geq 0.1$. In the present case,

\begin{equation}\label{west}
|w_0| +  \sum_{\{i,j\} \in \bracenom{n}{2}} |w_2(i,j)| \approx 1.93062 < 2.
\end{equation}

Set $T = 10000$. Then $T \cdot d_{min} \geq 1000$. Hence, by taking the largest zero of $J_1$ not exceeding 1000, which is $s_0 = 999.81148\ldots$, \eqref{phi_est} and  \eqref{west} imply that  for $t>T$,
\[
\varphi(t) > w_1(1) - 2 |J_0(s_0)| = 1.00892\ldots >1.
\]
Therefore, we have to find the minimum of $\varphi(t)$ on the interval $[0, T]$. This can be lower bounded by using interval arithmetic (note that $|J_0'(t)| = |J_1(t)|<1 $), which shows that
\[
\varphi(t) > 1 - \mu 
\]
with $\mu = 0.00006$ for $ t \in [0, T]$ (the global minimum of $\varphi(t)$ is attained at $t=3.77488\ldots$ with $\varphi(t) = 0.99995003\ldots$). 

In the final step, we divide all the coefficients $w$ by the factor $(1 - \mu)$. This ensures that $W(t) \geq 1$ holds for all $t \geq 0$, while the conditions $V(\eps) \geq 0$ remain valid. We may apply Proposition~\ref{prop1} to these transformed set of coefficients. Estimate \eqref{mwT} results in the upper bound 
\[
m_1(\R^2) \leq \frac{\widetilde{w}_T}{1 - \mu} = \frac{w_T + \nu}{ 1 - \mu } = 0.24699\ldots < 0.2470.
\]

\section{Implementation}
\label{sect_impl}

The linear programs we have to solve are quite large: the (LP) corresponding to the exhibited 23 vertex set $X_{23}$ has 25552 variables and 6099 constraints besides the positivity constraints. We used the Mosek \cite{Mosek} and Gurobi \cite{Gurobi} LP solvers, with typical solve times not exceeding 30 seconds on a regular laptop computer.

The beam search run for a week on 128 CPUs, creating and solving $186472$ linear programs, and resulted in a 24 element vertex set that satisfied our requirements. We observed that one of the points does not improve the estimate, therefore we removed it. We call the resulting set $X_{23}$.

We have used the NumPy computation package \cite{harris2020array} for numerical computation, and the SymPy computation package \cite{SymPy} for symbolical computation. We have parallelized the beam search algorithm via message passing. The final estimate was verified by a piece of independently developed Mathematica code that we publish at \cite{Web}.

Floating point numbers are much faster to work with than symbolic formulae. Hence, during the CPU-intense beam search, we only work with the former. After settling on a point set, we need to re-create it symbolically.

The algorithm that reverse-engineers a symbolic point set from a numerical point set proceeds as follows: It processes the points one by one, in the same order they were included in the set. Let us assume that it currently processes point $a$, which is provided as a complex floating point number. The symbolic expression of $a$ is found by one of the following three options:
\begin{enumerate}
\item If (up to numerical precision) $a$ is part of a parallelogram together with three already processed points $b, c, d$, then $a$ may be written up symbolically from the symbolic versions of $b, c, d$. 
\item  The same is true if a point $a$ is part of a regular triangle with unit sides together with already processed points $b$ and $c$. 

 \item  Otherwise, by the construction scheme, there exist formerly included points $b$ and $c$ both unit distance away from $a$. Using these, $a$ can be expressed as one of the two solutions of a quadratic equation in terms of the symbolic versions of $b$ and $c$. The appropriate solution is chosen according to the numerics.
\end{enumerate}
In principle, utilizing solely option (3) suffices to solve the reverse engineering task. However, in practice it leads to an exponential blow-up in formula size, making even moderately sized point sets ($n > 13$) unmanageable symbolically. Fortunately, options (2) and (3) are very rarely needed, since we construct point sets containing plenty of parallelograms. In fact, for the particular point set $X_{23}$ described below, one may express the vertices symbolically using only steps of type (1), except for a single instance of type (2). This yields that $X_{23}$ is a subset of the ring $\overline{R}_2$ spanned by the vertices of the Moser spindle (see \cite{PM22}). This is a surprising observation considering that our search was not restricted to this ring.

Having the algebraic expressions at hand,
\begin{enumerate}
\item we verify that the supposed unit distances are symbolically equal to one.
\item we verify that all congruence constraints \eqref{ieC} are valid.
\item we estimate the errors resulting from discrete approximation of the continuous (LP), in particular, ensure that the function $W(t)$ of \eqref{W_def} is at least 1 for all non-negative arguments.
\end{enumerate}
Our code may be found at  \cite{Web}, while the verification documentary is also available at \cite{arXiv}.

\section{Acknowledgements}

We are indebted to P. Ágoston, D. Czifra, P. Lakatos, D. Pálvölgyi, J. Parts, and F. Vallentin for valuable discussions, and for the anonymous referees for their useful remarks and suggestions.

\medskip

Research of G. A. was partially supported by ERC Advanced Grant "GeoScape" no.  882971,  by the Hungarian National Research grant no. NKFIH KKP-133819,  and by project no. TKP2021-NVA-09, which has been implemented with the support provided by the
Ministry of Innovation and Technology of Hungary from the National
Research, Development and Innovation Fund, financed under the
TKP2021-NVA funding scheme.
A. Cs., D. V. and P. Zs. were supported by the Ministry of Innovation and Technology NRDI Office within the framework of the Artificial Intelligence National Laboratory (RRF-2.3.1-21-2022-00004). A. Cs. was partly supported by the project TKP2021-NKTA-62 financed by the National Research, Development and Innovation Fund of the Ministry for Innovation and Technology, Hungary. M. M. was supported by the grants NKFIH-132097 and NKFIH-129335. 

\section{Statements and Declarations}
The authors declare that they have no conflict of financial or non-financial interests that are directly or indirectly related to the work submitted for publication.

\vspace{1 cm}
\noindent
{\sc Gergely Ambrus} ({\em corresponding author})
\smallskip

\noindent
{\em Department of Geometry, Bolyai Institute, University of Szeged, Aradi vért. tere 1, 6720, Szeged, Hungary, \\ and Alfréd Rényi Institute of Mathematics, Reáltanoda u. 13-15, 1053, Budapest, Hungary}

\smallskip

\noindent
e-mail address: \texttt{ambrus@renyi.hu}

\medskip
\noindent
{\sc Adrián Csiszárik}
\smallskip

\noindent
{\em Department of Computer Science, Institute of Mathematics, Eötvös Loránd University, Pázmány Péter sétány 1/C, H-1117, Budapest, Hungary, \\ and Alfréd Rényi Institute of Mathematics, Reáltanoda u. 13-15, 1053,  Budapest, Hungary}

\smallskip

\noindent
e-mail address: \texttt{csadrian@renyi.hu}

\medskip
\noindent
{\sc Máté Matolcsi}
\smallskip

\noindent
{\em Department of Analysis, Institute of Mathematics, Budapest University of Technology and Economics, Műegyetem rkp. 3., H-1111 Budapest, Hungary, \\ and Alfréd Rényi Institute of Mathematics, Reáltanoda u. 13-15, 1053, Budapest, Hungary}
\smallskip
\noindent
e-mail address: \texttt{matolcsi.mate@renyi.hu}

\medskip
\noindent
{\sc Dániel Varga}
\smallskip

\noindent
{\em Alfréd Rényi Institute of Mathematics, Reáltanoda u. 13-15, 1053,  Budapest, Hungary}

\smallskip

\noindent
e-mail address: \texttt{daniel@renyi.hu}

\medskip
\noindent
{\sc Pál Zsámboki}
\smallskip

\noindent
{\em Alfréd Rényi Institute of Mathematics, Reáltanoda u. 13-15, 1053, Budapest, Hungary}

\smallskip

\noindent
e-mail address: \texttt{zsamboki.pal@renyi.hu}

\newpage

\section{Appendix: the point set}
\vspace{-10 pt}
\begin{table}[hbt!]
\footnotesize
\begin{align*}
x_{1} &= 0 \\
x_{2} &= 1 \\
x_{3} &= \frac{1}{2} + \frac{\sqrt{3} i}{2} \\
x_{4} &= \frac{3}{2} + \frac{\sqrt{3} i}{2} \\
x_{5} &= \frac{5}{6} + \frac{\sqrt{11} i}{6} \\
x_{6} &= - \frac{\sqrt{33}}{12} + \frac{5}{12} + \frac{\sqrt{11} i}{12} + \frac{5 \sqrt{3} i}{12} \\
x_{7} &= - \frac{\sqrt{33}}{12} + \frac{5}{4} + \frac{5 \sqrt{3} i}{12} + \frac{\sqrt{11} i}{4} \\
x_{8} &= \frac{\sqrt{33}}{12} + \frac{13}{12} - \frac{\sqrt{11} i}{12} + \frac{\sqrt{3} i}{12} \\
x_{9} &= \frac{1}{4} + \frac{\sqrt{33}}{12} - \frac{\sqrt{11} i}{4} + \frac{\sqrt{3} i}{12} \\
x_{10} &= - \frac{\sqrt{33}}{12} + \frac{13}{12} - \frac{\sqrt{11} i}{12} - \frac{\sqrt{3} i}{12} \\
x_{11} &= \frac{2}{3} - \frac{\sqrt{11} i}{6} + \frac{\sqrt{3} i}{2} \\
x_{12} &= - \frac{\sqrt{33}}{12} + \frac{7}{12} - \frac{\sqrt{11} i}{12} + \frac{5 \sqrt{3} i}{12} \\
x_{13} &= - \frac{\sqrt{33}}{6} + \frac{5}{6} - \frac{\sqrt{3} i}{6} + \frac{\sqrt{11} i}{6} \\
x_{14} &= \frac{\sqrt{33}}{12} + \frac{7}{12} - \frac{\sqrt{11} i}{12} + \frac{7 \sqrt{3} i}{12} \\
x_{15} &= - \frac{\sqrt{33}}{12} + \frac{1}{12} - \frac{\sqrt{11} i}{12} - \frac{\sqrt{3} i}{12} \\
x_{16} &= - \frac{\sqrt{33}}{6} + \frac{2}{3} - \frac{\sqrt{11} i}{6} + \frac{\sqrt{3} i}{3} \\
x_{17} &= - \frac{\sqrt{33}}{12} + \frac{19}{12} - \frac{\sqrt{11} i}{12} + \frac{5 \sqrt{3} i}{12} \\
x_{18} &= \frac{1}{3} + \frac{\sqrt{11} i}{6} + \frac{\sqrt{3} i}{2} \\
x_{19} &= - \frac{\sqrt{33}}{12} + \frac{3}{4} - \frac{\sqrt{3} i}{12} + \frac{\sqrt{11} i}{4} \\
x_{20} &= - \frac{\sqrt{33}}{6} + 1 - \frac{\sqrt{3} i}{6} \\
x_{21} &= \frac{7}{6} - \frac{\sqrt{11} i}{6} \\
x_{22} &= - \frac{\sqrt{33}}{6} + \frac{4}{3} + \frac{\sqrt{11} i}{6} + \frac{\sqrt{3} i}{3} \\
x_{23} &= - \frac{\sqrt{33}}{4} + \frac{19}{12} - \frac{\sqrt{11} i}{12} + \frac{\sqrt{3} i}{4} \\
\end{align*}
\vspace{-20 pt}
\normalsize
    \caption{The point set $X_{23}$}
    \label{table:point set}
\end{table}

\begin{table}[hbt!]
$$
    \begin{array}{c|c|c|c}
    w_0 & w_T & w_1(1)  \\
    \hline
   0.378583312921677& 0.24697262945998308 &1.059383649998022
    \end{array}
    $$
    
    $$
    \begin{array}{c|c|c|c}
(i,j) & w_2(i,j) & (i,j) & w_2(i,j) \\
\hline
(1,10) & 0.014243384098949957  & (1,17) & 0.07397413460039694 \\
\hline
(1,23) & -0.008047304925038572  & (3,12) & 0.03487012105072677 \\
\hline
(3,15) & -0.0785963112357179  & (4,22) & 0.00022917246188142716 \\
\hline
(6,8) & -0.03025769554989927  & (6,15) & 0.018185030147879047 \\
\hline
(7,9) & -0.17935529642485845  & (7,12) & 0.08006137472171244 \\
\hline
(7,15) & 0.14034437164315525  & (7,21) & -0.017391357599152 \\
\hline
(7,23) & 0.09939574113576811  & (8,22) & -0.013665295941013265 \\
\hline
(8,23) & 0.017535950345541916  & (9,13) & -0.0667237004898899 \\
\hline
(9,22) & 0.02962214917215127  & (9,23) & -0.1543585559725672 \\
\hline
(11,22) & -0.050336630381192515  & (11,23) & 0.0395556964778143 \\
\hline
(12,19) & -0.01728864937672063  & (14,22) & -0.03251305281246628 \\
\hline
(14,23) & -0.07168501985142882  & (15,16) & 0.009019525024808076 \\
\hline
(21,22) & -0.08293311685123309  & (21,23) & -0.19185162418622392
\end{array}
    $$
    \caption{Nonzero dual coefficients of constraints other than (IEC)}
    \label{table:dual coefficients}
\end{table}

\end{document}